\documentclass[12pt]{amsart}
\usepackage{amsmath}
\usepackage{amssymb, amsfonts, amsthm}
\usepackage{mathtools}
\usepackage{amssymb}
\usepackage{ifthen}
\usepackage{graphicx}
\usepackage{float}
\usepackage{caption}
\usepackage{subcaption}
\usepackage{cite}
\usepackage{amsfonts}
\usepackage{amscd}
\usepackage{amsxtra}
\usepackage{color}
\usepackage{bm}

\newcommand{\IC}{{\mathbb C}}
\newcommand{\ID}{{\mathbb D}}

\newtheorem{theorem}{Theorem}[section]
\newtheorem{lemma}[theorem]{Lemma}
\newtheorem{corollary}[theorem]{Corollary}

\newtheorem{definition}[theorem]{Definition}

\theoremstyle{definition}
\newtheorem{remark}[theorem]{Remark}

\numberwithin{equation}{section}

\def\be{\begin{equation}}
	\def\ee{\end{equation}}

\newcounter{alphabet}
\newcounter{tmp}
\newenvironment{Thm}[1][]{\refstepcounter{alphabet}%
	\bigskip%
	\noindent%
	{\bf Theorem \Alph{alphabet}}%
	\ifthenelse{\equal{#1}{}}{}{ (#1)}%
	{\bf .} \itshape}{\vskip 8pt}

\newenvironment{Lem}[1][]{\refstepcounter{alphabet}%
	\bigskip%
	\noindent%
	{\bf Lemma \Alph{alphabet}}%
	{\bf .} \itshape}{\vskip 8pt}

\begin{document}

	\title[Asymptotic value of the multidimensional Bohr radius]
	{Asymptotic value of the multidimensional Bohr radius}

	\author[Vibhuti Arora]{Vibhuti Arora}
	\address{Vibhuti Arora, Department of Mathematics, National Institute of Technology Calicut, Kerala 673 601, India.}
	\email{vibhutiarora1991@gmail.com, vibhuti@nitc.ac.in}
	
	\author[Shankey Kumar]{Shankey Kumar}
	\address{Shankey Kumar, Department of Mathematics, Indian Institute of Technology Madras, Chennai, 600036, India.}
	\email{shankeygarg93@gmail.com}
	
	\author{Saminathan Ponnusamy}
	\address{ Saminathan Ponnusamy, Department of Mathematics, Indian Institute of Technology Madras, Chennai, 600036, India.}
	\address{Lomonosov Moscow State University, Moscow Center of Fundamental and Applied Mathematics, Moscow, Russia.}
	\email{samy@iitm.ac.in}

	
	\begin{abstract}
		This article determines the exact asymptotic value of the Bohr radii and the arithmetic Bohr radii for the holomorphic functions defined on  the unit ball of the $\ell_p^n$ space and having values in the simply connected domain of $\mathbb{C}$. Moreover, we investigate sharp Bohr radius for four distinct categories of holomorphic functions. These functions map the bounded balanced domain $G$ of a complex Banach space $X$ into the following domains: the right half-plane, the slit domain, the punctured unit disk, and the exterior of the closed unit disk.
	\end{abstract}
	
	\maketitle

	\section{Introduction and preliminaries}

The classical theorem of Harald Bohr \cite{Bohr14}, examined a century ago,  states that if the power series $\sum_{k=0}^{\infty}a_kz^k$ converges for $|z|<1$ and is bounded by $1$ in modulus, then the majorant series
$\sum_{k=0}^{\infty}|a_k|\, |z|^k$ is bounded by $1$ for $|z|\leq 1/3$ and the constant $1/3$ is optimal. This result, although looks simple, not only motivates many but also
generates intensive research activities to study analogues of this result in various setting--which we call Bohr's phenomena in modern language. This topic is interesting in its own right from the point of view
of analysis. In fact, the idea of Bohr that relates to power series was revived by many with great interest in the nineties due to the extensions to holomorphic functions of several complex variables and to more abstract settings.
For example in 1997, Boas and Khavinson \cite{BK97} defined $n$-dimensional Bohr radius for the family of holomorphic functions bounded by $1$ on the unit polydisk. This investigation led to Bohr type questions in different settings. In a series of papers, Aizenberg \cite{A00,A2005,A07}, Aizenberg et al. \cite{AAD}, Defant and Frerick \cite{DF}, and
Djakov and Ramanujan \cite{DjaRaman-2000} have established further results on Bohr's phenomena for multidimensional power series.
Several other aspects and generalizations of Bohr's inequality may be obtained from the literature.
In particular,  \cite[Section 6.4]{KM07} on Bohr's type theorems is useful in investigating new inequalities to holomorphic functions of several complex variables and more importantly to
solutions of partial differential equations.

	Let $X$ be a complex Banach space. For a domain $G\subset X$ and a domain $D\subset \mathbb{C}$, let $H(G,D)$ be the set of holomorphic mappings from $G$ into $D$. Every $f\in H(G,D)$ can be expanded, in a neighbourhood of any given $x_0 \in G$, into the series
	\begin{equation}\label{eqf}
		f(x)=\sum_{m=0}^{\infty}\cfrac{1}{m!}\,D^m f(x_0)\big(\underbrace{(x-x_0),(x-x_0),\dots,(x-x_0)}_{m-times}\big),
	\end{equation}
	where $D^mf(x_0)$, $m\in \mathbb{N}$, denotes the $m$-th Fr\'{e}chet derivative of $f$ at $x_0$, which is a bounded symmetric $m$-linear mapping from $\prod_{i=1}^m X$ to $\mathbb{C}$. For convenience, we may write $D^m f(x_0)\big((x-x_0)^m\big)$ instead of
	$$
	D^m f(x_0)\big(\underbrace{(x-x_0),(x-x_0),\dots,(x-x_0)}_{m-times}\big).
	$$
	It is understood that $D^0 f(x_0)\big((x-x_0)^0\big)=f(x_0)$. For additional information on this topic, we recommend to refer the book of Graham and Kohr \cite{GK}.

	
	A domain $G\subset X$ is said to be {\em balanced} if $zG\subset G$ for all $z\in \overline{\mathbb{D}}$, where $\mathbb{D}:=\{z\in \mathbb{C}: |z|<1\}$.
	Given a balanced domain $G$, we denote the higher dimensional Bohr radius by $K^G_X(D)$ which is the largest non-negative number $r$ with the property that the inequality
	\begin{equation*}\label{eqbi}
		\sum_{m=1}^{\infty}\bigg|\cfrac{1}{m!}\,D^m f(0)\big(x^m\big)\bigg|\le d(f(0),\partial D)
	\end{equation*}
	holds for all $x\in rG$ and for all holomorphic functions $f \in H(G,D)$ with the expansion \eqref{eqf} about the origin. Here $d(f(0),\partial D)$ denotes the Euclidean distance between $f(0)$ and the boundary $\partial D$ of the domain $D$. The definition of $K^G_X(D)$ is introduced by Hamada et al. \cite{HHK} (see also \cite{BD21}).  In this setting, Bohr's  remarkable result may be restated as $K^\mathbb{D}_\mathbb{C}(\mathbb{D})=1/3
	$. A relevant application of this outcome in the field of Banach algebras was provided by Dixon in \cite{D95}. After this, Bohr's result gained immense popularity among mathematicians. To see the progress on this topic, readers are encouraged to explore a comprehensive survey articles \cite{Abu-M16}, \cite[Chapter 8]{GarMasRoss-2018}, and the monograph \cite{DGMS19}. Additionally, for the latest developments pertaining to the Bohr phenomenon in the realm of complex variables, the references \cite{BDK04,BB04,Kaypon18,KM07,LLP2020} serve as valuable resources.
	
	The result that $K_{\mathbb{C}^n}^G(\mathbb{D})\ge 1/3$ was proved by Aizenberg \cite{A00} for any balanced domain $G$ in $\mathbb{C}^n$, and in the same paper he gave the optimal result $K^G_{\mathbb{C}^n}(\mathbb{D})= 1/3$ by assuming $G$ is convex. In 2009, Hamada et al. \cite[Corollary 3.2]{HHK} discussed a more general result by proving that $K^G_X(\mathbb{D})=1/3$ for any bounded balanced domain $G$ which is a subset of complex Banach space $X$. Also, by taking a restriction on $f\in H(G,\mathbb{D})$ such that $f(0)=0$ and $G\subset \mathbb{C}^n$, Liu and Ponnusamy \cite{LP21} improved the quantity $K^G_{\mathbb{C}^n}(\mathbb{D})\geq 1/\sqrt{2}$ and obtained a sharp radius $K^G_{\mathbb{C}^n}(\mathbb{D})=1/\sqrt{2}$ if $G$ is convex. Moreover, Bhowmik and Das \cite{BD21} calculated the following quantities:
	\begin{equation}\label{eqconn}
	\inf\{K^G_X(\Omega):\Omega\subset \mathbb{C} \mbox{ is simply connected}\}=3-2\sqrt{2}
	\end{equation}
	and
	\begin{equation}\label{eqconvex}
	\inf\{K^G_X(\Omega):\Omega\subset \mathbb{C} \mbox{ is convex}\}=\frac{1}{3}.
    \end{equation}
	In higher dimensions, there have been numerous significant and intriguing results (see for instance \cite{A00,A07}).
	
	
	To ensure clarity and facilitate understanding, let us fix a set of notations before we proceed any further.
	Let $\mathbb{H}:=\{z\in \mathbb{C}:{\rm Re}\,z>0\}$ be the right half-plane, $
	T:=\{z\in \mathbb{C}:|\arg z|<\pi\}
	$ be the slit region, $\mathbb{D}_0:=\{z\in \mathbb{C}:0<|z|<1\}$ be the punctured unit disk, and $\overline{\mathbb{D}}^c:=\{z\in \mathbb{C}:|z|>1\}$ be the exterior of the closed unit disk $\overline{\mathbb{D}}$. The spherical distance between two complex numbers of the extended complex plane is given by
	\begin{equation}\label{SD}
	\lambda (z,w)=
	\begin{cases}\vspace{0.1cm}
		\displaystyle\frac{|z-w|}{\sqrt{1+|z|^2}\sqrt{1+|w|^2}}& \text{if $z,w\in \mathbb{C}$,}\\\vspace{0.2cm}
		\displaystyle\frac{1}{\sqrt{1+|z|^2}} & \text{if $w=\infty$}.
	\end{cases}
	\end{equation}
	In \cite{AbuAli}, Abu-Muhanna and Ali observed that if $f(z)=\sum_{n=0}^{\infty}a_nz^n \in H(\mathbb{D},\overline{\mathbb{D}}^c)$ then
	\begin{equation}\label{EUD}
		\lambda\bigg(\sum_{n=0}^{\infty}|a_nz^n|, |a_0|\bigg)\le \lambda \left(f(0),\partial \overline{\mathbb{D}}^c\right)
	\end{equation}
	for all $|z|\le 1/3$, where the constant $1/3$ is sharp.
	In 2013, Abu-Muhanna and Ali in \cite{Abu4} obtained $K_\mathbb{C}^\mathbb{D}(\mathbb{H})=1/3$ and $K_\mathbb{C}^\mathbb{D}(T)=(\sqrt{2}-1)/(\sqrt{2}+1)$. Further, Abu-Muhanna et al. \cite{AAN17} determined $K_\mathbb{C}^\mathbb{D}(\mathbb{D}_0)=1/3$.
	 Therefore, it is natural to consider analogues of these results in a general setting.
	
	The following domains are the examples of the balanced domains, which we require in our further discussion:
	$$
	B_{\ell_p^n}:=\{z=(z_1,\dots,z_n)\in\mathbb{C}^n:\,\|z\|_p<1\}, \,\ 1\leq p \leq \infty,
	$$
	where the $p$-norm is given by $\|z\|_p=\left(\sum_{k=1}^{N}|z_k|^p\right)^{1/p}$ (with the obvious modification for $p=\infty$). Let $\mathbb{N}_0=\mathbb{N}\cup\{0\}$,
	$|\alpha|:=\alpha_1 +\cdots +\alpha_n$  and
	$z^\alpha:={z_1}^{\alpha_1}\cdots {z_n}^{\alpha_n}$ for
	$\alpha=(\alpha_1, \dots,\alpha_n)\in \mathbb{N}_0^n$
	(where $0^0$ is interpreted as $1$). 	
	Let $\Omega$ be a simply connected domain in $\mathbb{C}$. Then every $f\in H(B_{{\ell}_p^n},\Omega)$, $1\leq p\leq \infty$, can be expressed as
	$$
	f(z)=c_0+\sum_{|\alpha|\in \mathbb{N}} c_\alpha(f) z^\alpha, \quad \mbox{for }z\in B_{{\ell}_p^n}.
	$$
	It is clear from our previous discussion that if we replace $X$ by $\mathbb{C}^n$ and $G$ by $B_{\ell_p^n}$, then
	$$
	\sum_{|\alpha|= m} c_\alpha(f) z^\alpha = \cfrac{1}{m!}\,D^m f(0)\big(z^m\big), \quad \mbox{for }  z\in B_{{\ell}_p^n}.
	$$
	Unless otherwise stated throughout the article, we assume $p\in [1,\infty]$.
	Now, we prepare to give the following definition of the Bohr radii of the functions in $H(B_{{\ell}_p^n},\Omega)$.
	\begin{definition} For $1\leq p\leq \infty$, the Bohr radius $K_n^p(\Omega)$ is the supremum of all $r\in[0,1)$ such that
		\begin{equation}\label{eqBS}
			\sup_{z\in rB_{\ell_p^n}} \sum_{m=1}^{\infty}\sum_{|\alpha|=m} |c_\alpha(f) z^\alpha| \leq d(f(0),\partial \Omega)
		\end{equation}
		holds for all holomorphic functions $f \in H(B_{{\ell}_p^n},\Omega)$.
	\end{definition}
	\begin{remark}\label{HoPoRa} Let $f\in H(B_{{\ell}_p^n},\Omega)$.
		Consider $r>0$ such that
		$$
		\sum_{|\alpha|=m} |c_\alpha(f) (rz)^\alpha| \leq d(f(0),\partial \Omega)
		$$
		for all $m\in \mathbb{N}$.
		Then it follows easily that
		\begin{align*}
			\sum_{m=1}^{\infty} \sum_{|\alpha|=m} \left|c_\alpha(f) \left(\frac{r}{2}z\right)^\alpha\right| &=  \sum_{m=1}^{\infty} \frac{1}{2^m} \sum_{|\alpha|=m} \Big|c_\alpha(f) \left(rz\right)^\alpha\Big|\\
			&\leq d(f(0),\partial \Omega)\sum_{m=1}^{\infty} \frac{1}{2^m}=d(f(0),\partial \Omega)
		\end{align*}
		showing that $r/2 \leq K_n^p(\Omega)$.
	\end{remark}
	The above multidimensional generalizations of Bohr's theorem have garnered significant interest in recent years.
	The initial formulation of this generalization was presented by Boas and Khavinson \cite{B2000,BK97}. This generalization, in a more general setting, is studied in \cite{AAD}  in the spirit of functional analysis. Later, Defant et al. \cite{DFOOS11} obtained the optimal asymptotic estimate for $K_n^\infty(\mathbb{D})$ by using the fact that the Bohnenblust-Hille inequality is indeed hypercontractive. In the same year, the exact asymptotic value of $K_n^p(\mathbb{D})$ was given by  Defant and Frerick \cite{DF11}. In 2014, Bayart et al. \cite{BPS14} proved that $\lim_{n\to \infty}K_n^\infty(\mathbb{D})\sqrt{n}/(\sqrt{\log n})=1$. In the recent articles \cite{BCGMMS21,KM23} the lower bound of $K_n^p(\mathbb{D})$ was improved over the previously known lower bounds.
	Before we proceed further, define
	\begin{equation}\label{DBSCD}
		\tilde{K}_n^p:=
		\inf\{K_n^p(\Omega):\Omega\subset \mathbb{C} \mbox{ is simply connected}\} .
	\end{equation}
	Recently, in \cite{BD21}, authors estimate that $\lim_{n\to \infty}\tilde{K}_n^\infty\sqrt{n}/(\sqrt{\log n})=1$.
	
	
	The article is organized as follows. Our primary objective of this paper is to obtain the exact asymptotic bounds of the Bohr radius $\tilde{K}_n^p$, and to establish a lower bound of $K_n^p(\Omega)$. Our next objective is to study the arithmetic Bohr radii (definition is given in the next section) for the functions in $H(B_{{\ell}_p^n},\Omega)$, and to establish sharp bounds of $K^G_X(\mathbb{H})$, $K^G_X(\mathbb{D}_0)$, and
	$K^G_X(T)$, where $G$ represents a bounded balanced domain within a Banach space $X$. Additionally, we also determine a specific radius for which the inequality  \eqref{EUD} holds for functions belonging to the class $H(G,\overline{\mathbb{D}}^c)$.
	
	\section{Main Results}

Before stating the main results, we need to introduce some notation and definitions, and then present necessary preliminary tools.

	\subsection{Preliminaries on $m$-homogeneous polynomial} Let $\mathcal{P}(\prescript{m}{}{\ell_p^n})$ denote the space of all $m$-homogeneous polynomials defined on the $\ell_p^n$ space. Each polynomial
	$P\in \mathcal{P}(\prescript{m}{}{\ell_p^n})$ can be written in the form
	$$
	P(z_1,\dots,z_n)=\sum_{\alpha\in\Lambda(m,n)} a_\alpha z^\alpha,
	$$
	where $\Lambda(m,n):=\{\alpha\in\mathbb{N}_0^n:\,|\alpha|=m\}$ and $a_\alpha \in X$. Equivalently, we can express it in another form as
	$$
	P(z_1,\dots,z_n)=\sum_{{\bf j}\in\mathcal{J}(m,n)} c_{\bf j} z_{\bf j},
	$$
	where $\mathcal{J}(m,n):=\{{\bf j}=(j_1,\dots,j_m):\,1\leq j_1\leq\dots \leq j_m\leq n\}$, $z_{\bf j}:=z_{j_1}\cdots z_{j_m}$, and $c_{\bf j}\in X$. Then the elements $(z^\alpha)_{\alpha\in\Lambda(m,n)}$ or $(z_{\bf j})_{{\bf j}\in\mathcal{J}(m,n)}$ are referred to as the monomials. Note that the coefficients $c_{\bf j}$ and $a_\alpha$ are related by the relation  $c_{\bf j}=a_\alpha$ with ${\bf j}=(1,\overset{\alpha_1}{\dots},1,\dots,n,\overset{\alpha_n}{\dots},n)$. More precisely there is a one-to-one correspondence between two index sets $\Lambda(m,n)$ and $\mathcal{J}(m,n)$.
	
	
	The following result due to Bayart et al. \cite[Theorem 3.2]{BDS19} plays a key role in the proof of Theorem \ref{maintheorem}.

	\begin{Thm}\label{imtheorem} {\rm(\cite[Theorem 3.2]{BDS19})}
		Let $m\geq1$ and $P\in \mathcal{P}(\prescript{m}{}{\ell_p^n})$ be of the form $P(z)=\sum_{{\bf j}\in\mathcal{J}(m,n)} c_{\bf j} z_{\bf j}$. Then, for every $u\in \ell_p^n$, we have
		$$
		\sum_{{\bf j}\in \mathcal{J}(m,n)} |c_{\bf j}| \, |u_{\bf j}| \leq C(m,p)  |\mathcal{J}(m-1,n)|^{1-\frac{1}{\min\{p,2\}}} \|u\|_p^m \,\|P\|_{\mathcal{P}(\prescript{m}{}{\ell_p^n})},
		$$
		where
		$$
		C(m,p)\leq \begin{cases}\vspace{0.1cm}
			em e^{(m-1)/p}
			& \text{if $1\leq p \leq 2$,}\\\vspace{0.2cm}
			em 2^{(m-1)/2}
			& \text{if $2\leq p \leq \infty$}.
		\end{cases}
		$$	
	\end{Thm}

\subsection{The notion of subordination connected with $f\in H(B_{{\ell}_p^n},\Omega)$}
In order to proceed to state a couple of main results, we require the following well-known definitions:
	
\begin{enumerate}
	\item[(i)] {\bf Subordination:} Suppose $f$ and $g$ are analytic functions in $\mathbb{D}$. We say that $g$ is {\em subordinate} to $f$, written by $g\prec f$, or  $g(z) \prec f(z)$, if there is an analytic function
	$w:\,\mathbb{D}\rightarrow \mathbb{D}$ with $w(0)=0$ and such that $g=f\circ w$.
	
	\item[(ii)] {\bf Starlike functions:} A domain $D$ in $\IC$ is called starlike with respect to $0$ (or simply starlike) if $tw\in D$ whenever $w\in D$ and $t\in [0,1]$.
A univalent function $f$ is said to be starlike in $\ID$ if it  maps $\mathbb{D}$ onto a domain that starlike and $f(0)=0$. Analytically, this holds if and only if
$f'(0)\neq 0$ and ${\rm Re}\, (zf'(z)/f(z))>0$ for $z\in \ID$.
	
	\item[(iii)] {\bf Convex functions:} A  domain $D$ in $\IC$ is called convex if $(1-t)w_1+tw_2\in D$ whenever $w_1, w_2\in D$ and $t\in [0,1]$.  An analytic function is said to be convex
if it  maps $\mathbb{D}$ univalently onto a convex domain. It is well-known that $f$ is convex if and only if $g=zf'$ is starlike.
\end{enumerate}
The following well-known result of Rogosinski  (see \cite[p. 195, Theorem 6.4]{Duren83}) on subordination  serves as a useful tool to prove Theorem \ref{maintheorem}.

\begin{Lem}\label{lemmasub}
{\rm \cite[Theorem 6.4]{Duren83}}
	Let $f(z)=\sum_{n=1}^{\infty}a_nz^n$ and $g(z)=\sum_{n=1}^{\infty}b_nz^n$ be two analytic functions in $\mathbb{D}$ such that  $f\prec g$. 
Then
	\begin{itemize}
		\item[(i)] $|a_n|\le n|b_1|$ if $g$ is starlike univalent in $\mathbb{D}$,
		\item[(ii)] $|a_n|\le |b_1|$ if $g$ is convex univalent in $\mathbb{D}$.
	\end{itemize}
\end{Lem}
The inequality in Lemma B(i) 
continues to hold even if $g$ is just univalent in $\mathbb{D}$. This fact follows from the work of de Branges \cite{DeB1} in 1985.
Also, we use the following well-known results: if $g$ is the univalent  mapping from $\mathbb{D}$ onto $\Omega$, then (cf. \cite[Corollary 1.4]{Pomm92})
\begin{equation}\label{USEQ}
	|g'(0)|\geq d(g(0),\partial\Omega)\geq \frac{1}{4}	|g'(0)|,
\end{equation}
and similarly, if $g$ is convex also, then
\begin{equation}\label{CSEQ}
	|g'(0)|\geq d(g(0),\partial\Omega)\geq \frac{1}{2}	|g'(0)|.
\end{equation}
	 The proofs of Theorems \ref{maintheorem} and \ref{KA-1} which  used the inequalities \eqref{subord} and \eqref{subord-2} require some discussion. Now, we begin the discussion by considering $f\in H(B_{{\ell}_p^n},\Omega)$. For a fixed $z\in B_{{\ell}_p^n}$, consider the function $F$ on $\mathbb{D}$ defined by
	\begin{equation*}
		F(y):=f(zy),\quad y\in \mathbb{D}.
	\end{equation*}
	Then $F:\mathbb{D}\to \Omega$ is holomorphic and
	\begin{equation*}
		F(y)=f(0)+\sum_{m=1}^{\infty}\Bigg(\sum_{|\alpha|=m} c_\alpha(f) z^\alpha \Bigg) y^m.
	\end{equation*}
	Therefore, $F \prec g$ in $\mathbb{D}$, where
	$g$ is the univalent mapping from $\mathbb{D}$ onto $\Omega$ satisfying $g(0)=F(0)=f(0)$. Then, using Lemma B(i) 
 and the inequality \eqref{USEQ}, we have
	\begin{equation}\label{subord}
		\bigg|\sum_{|\alpha|=m} c_\alpha(f) z^\alpha\bigg|\leq 4md(f(0),\partial \Omega), \quad \mbox{for each $z\in B_{{\ell}_p^n}$}.
	\end{equation}
	If $\Omega$ is a convex domain, then  Lemma B(ii) 
and the inequality \eqref{CSEQ} provide
	\begin{equation}\label{subord-2}
		\bigg|\sum_{|\alpha|=m} c_\alpha(f) z^\alpha\bigg|\leq 2 d(f(0),\partial \Omega), \quad \mbox{for each $z\in B_{{\ell}_p^n}$}.
	\end{equation}

	
	\subsection{Asymptotic behaviour of the Bohr radius}
Regarding the multi-dimensional Bohr radius, significant contributions have been made by a number of authors (see \cite{A00, B2000, BK97, DF11, DT89}). Their results and ideas reach  up to the following
optimal  result for $ K_n^p(\mathbb{D})$.

\begin{Thm}\label{MDBR}
	There exists a constant $D\geq 1$ such that for each $1\leq p \leq \infty$ and all $n$
	$$
	\frac{1}{D} \Bigg(\frac{\log n}{n}\Bigg)^{1-\frac{1}{\min\{p,2\}}}\leq K_n^p(\mathbb{D}) \leq D \Bigg(\frac{\log n}{n}\Bigg)^{1-\frac{1}{\min\{p,2\}}}.
	$$
\end{Thm}

\begin{remark}
	Note that, in Theorem C, 
the lower bound of $K_n^p(\mathbb{D}) $ was established by Defant and Frerick \cite{DF11}, while the upper bound for the case $p=\infty$
	was obtained by Boas and Khavinson \cite{BK97}. For the remaining values of $p$, the upper bound was determined by Boas \cite{B2000}.	
\end{remark}

	Now we are in a position to state and prove our first main result which
	gives bounds of the Bohr radii $\tilde{K}_n^p$ defined by \eqref{DBSCD}. In fact, in the case of simply connected domain, we show that the left hand side inequality  holds
with $B$ in place of $1/D$, where $B_0=1/(8e^{5})\leq B\leq 1  $.
	
	\begin{theorem} \label{maintheorem}
		There are positive constants $B$ and $D$ such that for each $1\leq p \leq \infty$ and all $n$
		$$
		B \Bigg(\frac{\log n}{n}\Bigg)^{1-\frac{1}{\min\{p,2\}}}\leq \tilde{K}_n^p \leq D  \Bigg(\frac{\log n}{n}\Bigg)^{1-\frac{1}{\min\{p,2\}}}.
		$$
	\end{theorem}
	\begin{proof}
		We begin the discussion with the following simple observation from \eqref{DBSCD}:
		$\tilde{K}_n^p\leq K_n^p(\mathbb{D})$.
		Then, from Theorem C, 
we have
		$$
		K_n^p(\mathbb{D})\leq D \Bigg(\cfrac{\log n}{n}  \Bigg)^{1-\frac{1}{\min\{p,2\}}}
		$$
		for some $D\geq 1$. Now, we proceed to obtain a lower estimate.
		To do this, first we fix $P\in \mathcal{P}(\prescript{m}{}{\ell_p^n})$, where  $P(z)=\sum_{{\bf j}\in\mathcal{J}(m,n)} c_{\bf j} z_{\bf j}$,  and $u \in \ell_p^n$. Then, by Theorem A 
and the inequality \eqref{subord}, we obtain
		$$
		\sum_{{\bf j}\in \mathcal{J}(m,n)} |c_{\bf j} (P)|\, |u_{\bf j}| \leq 4 m \, C(m,p)  |\mathcal{J}(m-1,n)|^{1-\frac{1}{\min\{p,2\}}} \|u\|_p^m d(f(0),\partial \Omega).
		$$
		If we use the fact that the number of elements in the index set $\mathcal{J}(m-1,n)$ is
		$$
		\frac{(n+m-2)!}{(n-1)!(m-1)!},
		$$
		then we get
		$$
		\sum_{{\bf j}\in \mathcal{J}(m,n)} |c_{\bf j} (P)| \, |u_{\bf j}| \leq 4m \,C(m,p)  \Bigg(\frac{(n+m-2)!}{(n-1)!(m-1)!}\Bigg)^{1-\frac{1}{\min\{p,2\}}} \|u\|_p^m \ d(f(0),\partial \Omega).
		$$
		It is not difficult to observe that
		$$
		\frac{(n+m-2)!}{(n-1)!(m-1)!} \leq \frac{(n+m)^{m-1}}{(m-1)!} \leq \frac{e^m}{m^{m-1}} (n+m)^{m-1}=e^m \bigg(1+\frac{n}{m} \bigg)^{m-1},
		$$
		which leads to
		$$
		\sum_{{\bf j}\in \mathcal{J}(m,n)} |c_{\bf j} (P)| \, |u_{\bf j}|\leq 4m\, C(m,p)\ \bigg(1+\frac{n}{m} \bigg)^{(m-1)\Big(1-\frac{1}{\min\{p,2\}}\Big)} e^{m\left(1-\frac{1}{\min\{p,2\}}\right)}  \ \|u\|_p^m \ d(f(0),\partial \Omega).
		$$
		A simple observation shows that
		\begin{equation*}
			\bigg(1+\frac{n}{m} \bigg)^{\frac{m-1}{m}} \leq 2^{\frac{m-1}{m}} \max \bigg\{1, \bigg(\frac{n}{m} \bigg)^{\frac{m-1}{m}} \bigg\} \leq 2 \max \bigg\{1, \frac{m^{1/m}n}{m n^{1/m}} \bigg\}.
		\end{equation*}
		Note that $x\mapsto xn^{1/x}$ is decreasing from $x=0$ to $x=\log n$ and increasing thereafter. Thus,
		$$
		\bigg(1+\frac{n}{m} \bigg)^{\frac{m-1}{m}} \leq \frac{2n}{\log n}
		$$
		so that
		$$
		\sum_{{\bf j}\in \mathcal{J}(m,n)} |c_{\bf j} (P)| \, |u_{\bf j}|\leq 4m \,C(m,p)\ \bigg(\frac{2n}{\log n} \bigg)^{m\Big(1-\frac{1}{\min\{p,2\}}\Big)} e^{m-\frac{m}{\min\{p,2\}}}  \ \|u\|_p^m \ d(f(0),\partial \Omega).
		$$
		Finally, using Remark \ref{HoPoRa} we obtain
		$$
		K_n^p(\Omega)\geq B \Bigg(\cfrac{\log n}{n}  \Bigg)^{1-\frac{1}{\min\{p,2\}}}
		$$
		for some $B>0$. Indeed, it can be noted that $B_0=1/(8e^{5})\leq B\leq 1  $.
		This concludes the proof.
	\end{proof}

    \subsection{Asymptotic behaviour of the arithmetic Bohr radius}
	The arithmetic Bohr radius was for the first time introduced and studied by Defant et al. in \cite{DMP08}. The concept of the arithmetic Bohr radius plays an important role in studying the upper bound of the multi-dimensional Bohr radius and is also used as a main tool to derive upper inclusions for domains of convergence  (see \cite{DMP09}).
	\begin{definition}
		The arithmetic Bohr radius $A_n^p(\Omega)$, $1\leq p \leq \infty$, is defined as
		\begin{align*}
			A_n^p(\Omega):= \sup\bigg\{&\frac{1}{n}\sum_{i=1}^{n}r_i, r\in \mathbb{R}_{\geq 0}^n:\sum_{m=1}^{\infty}\sum_{|\alpha|=m} |c_\alpha(f)| r^\alpha \leq d(f(0),\partial \Omega) \\
			&\mbox{ for all } f\in H(B_{{\ell}_p^n},\Omega)\bigg\},
		\end{align*}
		where $\mathbb{R}_{\geq 0}^n=\{r=(r_1,\dots,r_n)\in \mathbb{R}^n:\, r_i\geq 0, 1\leq i \leq n \}$. Also, we set
		$$
		\tilde{A}_n^p=\inf\{A_n^p(\Omega):\Omega\subset \mathbb{C} \mbox{ is simply connected}\}.
		$$
	\end{definition}

	\begin{theorem}\label{LSABR} Let $2\leq p\leq \infty$. Then we have
		$$
		\limsup_{n\to \infty} A_n^p(\mathbb{D}) \frac{\displaystyle n^{ \big(\frac{1}{2}+\frac{1}{\max\{p,2\}}\big)}}{  \sqrt{\log n}}\leq1.
		$$
	\end{theorem}
	\begin{proof}
		To achieve our goal, we use the Kahane-Salem-Zygmund inequality (see \cite{B2000,BK97,DMP08}) which states that for $n,m\geq 2$, there exist coefficients $(c_\alpha)_{|\alpha|=m}$ with $|c_\alpha|=m!/\alpha!$ for all $\alpha$ such that
		$$
		\sup_{z\in B_{\ell_p^n}} \bigg|\sum_{|\alpha|=m} c_\alpha z^\alpha\bigg| \leq \sqrt{32m\log(6m)} n^{ \frac{1}{2}+\big(\frac{1}{2}-\frac{1}{\max\{p,2\}}\big)m} \sqrt{m!}=:L.
		$$
		Let $r=(r_1,\dots,r_n) \in \mathbb{R}_{\geq 0}^n$ such that
		$$
		\sum_{m=1}^{\infty} \sum_{|\alpha|=m} |a_\alpha(f)|r^{\alpha}\leq d(f(0),\partial\mathbb{D}) \quad
		\mbox{for all  $f\in H(B_{{\ell}_p^n},\mathbb{D})$}.
		$$
		Then
		$$
		\frac{1}{L}\bigg(\sum_{i=1}^{n}r_i\bigg)^m=\frac{1}{L}\sum_{|\alpha|=m} \frac{m!}{\alpha!} r^\alpha=\frac{1}{L}\sum_{|\alpha|=m} |c_\alpha| r^\alpha \leq 1.
		$$
		This gives that
		$$
		\bigg(\sum_{i=1}^{n}r_i\bigg)^m\leq \sqrt{32m \log(6m)} n^{ \frac{1}{2}+\big(\frac{1}{2}-\frac{1}{\max\{p,2\}}\big)m} \sqrt{m!}.
		$$
		Consequently,
		$$
		\frac{1}{n}\sum_{i=1}^{n}r_i \leq \big(32mn \log(6m)\big)^{\frac{1}{2m}} \frac{(m!)^{\frac{1}{2m}}}{n^{ \big(\frac{1}{2}+\frac{1}{\max\{p,2\}}\big)}}.
		$$
		Using Stirling's formula $\left(m!\leq \sqrt{2\pi m}e^{\frac{1}{12m}}m^me^{-m}\right)$, one obtains
		\begin{equation*}
			\frac{1}{n}\sum_{i=1}^{n}r_i \leq \big(32mn \log(6m)\big)^{\frac{1}{2m}} \left( \sqrt{2\pi m}e^{\frac{1}{12m}}e^{-m}\right)^{\frac{1}{2m}} \frac{\sqrt{m}}{n^{ \big(\frac{1}{2}+\frac{1}{\max\{p,2\}}\big)}}.  \end{equation*}
		Letting $m=\lfloor \log(n) \rfloor$, we have
		\begin{equation*}
			\left(\frac{\displaystyle n^{ \big(\frac{1}{2}+\frac{1}{\max\{p,2\}}\big)}}{  \sqrt{\log n}}\right)\frac{1}{n}\sum_{i=1}^{n}r_i \leq \left(32  \sqrt{2\pi }\lfloor \log(n) \rfloor^{\frac{3}{2}}e^{\frac{1}{12\lfloor \log(n) \rfloor}}\log(6\lfloor \log(n) \rfloor)\right)^{\frac{1}{2\lfloor \log(n) \rfloor}} \frac{n^{\frac{1}{2\lfloor \log(n) \rfloor}}}{\sqrt{e}}.  \end{equation*}
		It is easy to see that
		$$
		\lim_{n\to \infty} \left(32  \sqrt{2\pi }\lfloor \log(n) \rfloor^{\frac{3}{2}}e^{\frac{1}{12\lfloor \log(n) \rfloor}}\log(6\lfloor \log(n) \rfloor)\right)^{\frac{1}{2\lfloor \log(n) \rfloor}}=1.
		$$
		Further, the observation
		$$
		\frac{1}{2} \leq {\frac{\log n}{2\lfloor \log(n) \rfloor}} \leq {\frac{\lfloor \log n \rfloor+1}{2\lfloor \log(n) \rfloor}}
		$$
		gives that
		$$
		\lim_{n\to \infty} n^{\frac{1}{2\lfloor \log(n) \rfloor}}=\sqrt{e}.
		$$
		Both of the above observations conclude the result.
	\end{proof}
The proof of the next lemma follows exactly as in \cite[Lemma 4.3]{DMP08}. So, we skip the details here.
\begin{lemma}\label{AKlemma}
	For $n\in \mathbb{N}$ and $1\leq p\leq \infty$, we have
	$$
	A_n^p(\Omega)\geq \frac{K_n^p(\Omega)}{n^{1/p}}.
	$$
\end{lemma}
 For our further discussion we need the following result which is obtained in \cite[Theorem 4.1]{DMP08} and \cite[Remark 2]{DF11}.
	
\begin{Thm}\label{ABR}
	There is an absolute constant $D\geq 1$ such that for each $1\leq p\leq \infty$ and
	each $n$
	$$
	\frac{1}{D} \frac{  \big( \log n\big)^{1-(1/\min\{p,2\})}}{\displaystyle n^{\frac{1}{2}+(1/\max\{p,2\})}}\leq A_n^p(\mathbb{D}) \leq D  \frac{  \big( \log n\big)^{1-(1/\min\{p,2\})}}{\displaystyle n^{\frac{1}{2}+(1/\max\{p,2\})}}.
	$$
\end{Thm}
The lower bound in the following result can simply be obtained using Theorem \ref{maintheorem} and Lemma \ref{AKlemma}. The upper bound follows from Theorem D 
and the observation that
$$
\tilde{A}_n^p\leq A_n^p(\mathbb{D}).
$$

\begin{theorem}\label{ABRS}
	There are positive constants $B$ and $D$ such that for each $1\leq p \leq \infty$ and all $n$
	$$
	B \frac{  \big( \log n\big)^{1-(1/\min\{p,2\})}}{\displaystyle n^{\frac{1}{2}+(1/\max\{p,2\})}}\leq \tilde{A}_n^p \leq D  \frac{  \big( \log n\big)^{1-(1/\min\{p,2\})}}{\displaystyle n^{\frac{1}{2}+(1/\max\{p,2\})}}.
	$$
\end{theorem}

\begin{remark}
	For instance $B$ and $D$ in Theorem \ref{ABRS}
can be taken from Theorems \ref{maintheorem} and D,
 respectively.
\end{remark}
The next result gives the asymptotic behaviour of $\tilde{A}_n^\infty$.
\begin{theorem}
	$\lim\limits_{n\to \infty} \tilde{A}_n^\infty \sqrt{n/\log(n)}=1$.
\end{theorem}
\begin{proof}
	It is already obtained in \cite[Theorem 2]{BD21} that for every $\epsilon \in (0,1/2)$
	$$
	K_n^\infty(\Omega)\geq (1-2 \epsilon) \sqrt{\frac{\log n}{n}},
	$$
	for large enough $n$. By Lemma \ref{AKlemma}, for every $\epsilon \in (0,1/2)$, we have
	$$
	A_n^\infty(\Omega)\geq (1-2 \epsilon) \sqrt{\frac{\log n}{n}}
	$$
	for large enough $n$. The upper estimation can be obtained using Theorem \ref{LSABR}  and the fact that $
	\tilde{A}_n^\infty\leq A_n^\infty(\mathbb{D})$. This completes the proof.
	\end{proof}
	
	\subsection{Bounds on $K_{n}^p(\Omega)$ and $A_{n}^p(\Omega)$ through Sidon constant}
	We denote $\partial B_{\ell_p^n}=\{z\in \mathbb{C}^n:\|z\|_p =1\}$. Also, we define the norms $\|.\|_\infty$ and $\|.\|_1$ on the elements of the space $\mathcal{P}(\prescript{m}{}{\ell_p^n})$ as
	$$
	\|P\|_\infty:= \sup_{z\in \partial B_{\ell_p^n}}|P(z)| \
	\mbox{ and } \
	\|P\|_1:=\sup_{z\in \partial B_{\ell_p^n}} \sum_{|\alpha|=m} |c_\alpha|\,|z^\alpha|,
	$$
	respectively.
	It is not hard to observe that $\|P\|_\infty \leq \|P\|_1$.
	
	
	\begin{definition}\label{sidon}
		For each pair $n,m\in \mathbb{N}$ and each $1\leq p\leq \infty$, we define the constant $S_p(m,n)$ as
		\begin{align*}
			S_p(m,n)=\inf\{M>0:\|P\|_1\leq M\|P\|_\infty \text{ for all } P\in \mathcal{P}(\prescript{m}{}{\ell_p^n})\}.
		\end{align*}
		For the case $p=\infty$, the constant $S_p(m,n)$ known as the Sidon constant denoted by $S(m,n)$.
	\end{definition}
From the above definition, it is clear that  $S_p(m,n)\geq 1$ for every $m,n$ and $p$. Moreover, if $P(z)=\sum_{j=1}^{n}c_j z_j$, then choose $t_{z_j}\in [0,2\pi]$ such that $|c_j z_j|=c_j z_j e^{it_{z_j}}$ for $j=1,\dots,n$. We obtain
\begin{equation*}
	\sum_{j=1}^{n}|c_j z_j|= \sum_{j=1}^{n}c_j z_j e^{it_{z_j}}
	\leq \|P\|_\infty
\end{equation*}
which shows that $S_p(1,n)= 1$ for every $n$ and $p$.	
	
	
	\begin{theorem}\label{KA-1}
		Let $n\in\mathbb{N}\backslash \{1\}$ and $1\leq p\leq \infty$. If $\Omega$ is a simply connected domain and  $H_{n}^p \in(0,1)$ is the solution of the following equation
		\begin{align} \label{eqsd}
			x+\sum_{m=2}^{\infty}m S_p(m,n) x^m= \frac{1}{4},
		\end{align}		
		then
		\begin{equation}\label{lbmbr}
		K_{n}^p(\Omega) \geq H_{n}^p \quad \mbox{and}  \quad  A_{n}^p(\Omega)\geq \frac{H_{n}^p}{n^{1/p}}.
		\end{equation}
		If $\Omega$ is convex, then the number $H_{n}^p$  in \eqref{lbmbr} can be replaced by the number which is the solution of the following equation
		\begin{align*}
			x+\sum_{m=2}^{\infty} S_p(m,n) x^m= \frac{1}{2}.
		\end{align*}		
	\end{theorem}
	\begin{proof}
		Assume that $\Omega$ is simply connected domain. Now, using Definition \ref{sidon} and inequality \eqref{subord}, we have
		$$
		\sum_{|\alpha|=m} |c_\alpha(f) (rz)^\alpha| \leq 4mS_p(m,n)r^md(f(0),\partial \Omega)
		$$
		and thus,
		$$
		\sum_{m=1}^{\infty}\sum_{|\alpha|=m} |c_\alpha(f) (rz)^\alpha| \leq 4d(f(0),\partial \Omega)\sum_{m=1}^{\infty}mS_p(m,n)r^m.
		$$
		We have already observed that $S_p(1,n)=1$. Then the equation \eqref{eqsd} shows that
		$$
		\sum_{m=1}^{\infty}\sum_{|\alpha|=m} |c_\alpha(f) (rz)^\alpha| \leq d(f(0),\partial \Omega) \quad \mbox{for every $r\leq H_{n}^p$}.
		$$
		It is easy to observe that $f(x):=x+\sum_{m=2}^{\infty}m S_p(m,n) x^m$ is an increasing function of $x$ for $x\geq 0$, $f(0)=0<1/4$ and $f(1/2)>1/4$. This concludes that $ H_{n}^p \in (0,1)$.
		
		 Using the above approach and the inequality \eqref{subord-2} we can obtain the desired result when $\Omega$ is convex. Finally, the lower bound of $A_{n}^p(\Omega)$ can be obtained easily using Lemma \ref{AKlemma} and the lower bound of $K_{n}^p(\Omega)$. This completes the proof.
	\end{proof}
	
	
	\subsection{Multidimensional Bohr's radius for Half-plane and Convex domains}
	Let $\mathcal{P}$ denote the class of all analytic functions $p$ with $p(0)=1$ having positive real part in $\mathbb{D}$.
	The following fundamental result is well-known (cf. \cite[Corollary 2.3]{Pomm75})  and the Carath\'{e}odory Lemma  (cf. \cite[p. 41]{Duren83}).
	
	\begin{Lem} \label{lemmap}
		For every $p\in \mathcal{P}$ of the form $p(z)=1+c_1z+c_2z^2+\cdots$, we have $|c_n|\le 2$ for $n\ge 1$.
	\end{Lem}
	
	
	We are now ready to state and prove our initial result concerning the multidimensional Bohr's phenomenon for functions belonging to the class $H(G,\mathbb{H})$.
	
	
	\begin{theorem}\label{ThmH}
		Let $G$ be a bounded balanced domain and $f\in H(G,\mathbb{H})$ be of the form \eqref{eqf} in a neighbourhood of the origin. Then $K^G_X(\mathbb{H})=1/3$.
	\end{theorem}
	
	\begin{proof}
		 Using \eqref{eqconvex}, we can easily establish that $K^G_X(\mathbb{H})\geq 1/3$. However, we provide an alternative proof for the lower bound of $K^G_X(\mathbb{H})$ here. Note that ${\rm Re}f(0)>0$, for $f\in H(G,\mathbb{H})$. For any fixed $y\in G$, we introduce $F$ by
		\begin{equation*}
			F(z):=f(zy),\quad z\in \mathbb{D}.
		\end{equation*}
		Then $F:\mathbb{D}\to \mathbb{H}$ is holomorphic, $F(0)=f(0)$ and
		\begin{equation}\label{eqHF}
			F(z)=f(0)+\sum_{m=1}^{\infty}\cfrac{1}{m!}\,D^m f(0)\big(y^m\big)z^m.
		\end{equation}
		Define $g\in \mathcal{P}$ by
		\begin{equation*}
			g(z)=\cfrac{F(z)-i{\rm Im}(f(0))}{{\rm Re}(f(0))}=1+\cfrac{1}{{\rm Re}(f(0))}\sum_{m=1}^{\infty}\cfrac{1}{m!}\,D^m f(0)\big(y^m\big)z^m.
		\end{equation*}
		By Lemma~E, 
we obtain
		\begin{equation*}
			\bigg|	\cfrac{\,D^m f(0)\big(y^m\big)}{m!}\bigg| \le 2 {\rm Re}(f(0))
		\end{equation*}
		for $m\ge 1$. Thus, for $x\in (1/3)G$, we have by the last inequality that
		\begin{equation*}
			\sum_{m=1}^{\infty}\bigg|	\cfrac{\,D^m f(0)\big(x^m\big)}{m!}\bigg|=\sum_{m=1}^{\infty}\bigg|	\cfrac{\,D^m f(0)\big(y^m\big)}{m!}\bigg| \bigg(\frac{1}{3}\bigg)^m\le
 {\rm Re}(f(0)) =d(f(0),\partial \mathbb{H}).
		\end{equation*}
		
		In order to prove that the constant $1/3$ is optimal, we use the technique given in \cite{HHK}. For each $1/3<r_0<1$, there exists a $c\in(0,1)$ and $v\in \partial G$ such that $cr_0>1/3$ and $c\sup_{x\in \partial G}||x||<||v||$. Next, we consider the function $f_0$ on $G$ by
		$$
		f_0(x)=L\bigg(\cfrac{c\psi_v(x)}{||v||}\bigg),
		$$
		where $L(z)=(1+z)/(1-z),\mbox{ for }z\in \mathbb{D}$, $\psi_v$ is a bounded linear functional on $X$ with $\psi_v(v)=||v||$ and $||\psi_v||=1$. Clearly, $c\psi_v(x)/||v||\in\mathbb{D}$ and $f_0\in H(G,\mathbb{H})$. Choosing $x=r_0v$ gives
		\begin{equation*}
			f_0(r_0v)=\cfrac{1+cr_0}{1-cr_0}=1+2\sum_{n=1}^{\infty}{(cr_0)}^n=f_0(0)+\sum_{m=1}^{\infty}\bigg|	\cfrac{\,D^m f_0(0)\big(x^m\big)}{m!}\bigg| >2=1+d(f_0(0),\partial\mathbb{H}).
		\end{equation*}
		The proof of the theorem is complete.
	\end{proof}
	
	
	It is now appropriate to remark that when $X=\mathbb{C}$ and $G=\mathbb{D}$, Theorem \ref{ThmH} coincides with \cite[Theorem 2.1]{Abu4}.
	Furthermore, a subsequent result can be derived from Theorem \ref{ThmH}, which pertains to holomorphic functions defined on a bounded balanced domain and taking values in a convex domain.
	
	\begin{corollary}
		Let $G$ be a bounded balanced domain and $f\in H(G,C)$, where $C$ is a convex domain, be of the form \eqref{eqf} in a neighbourhood of the origin. Then $K^G_X(C)=1/3$.
	\end{corollary}
	
	\begin{proof} This corollary has been proved by Bhowmik and Das \cite[Theorem 1]{BD21}. It might be appropriate to indicate an alternate proof as a consequence of the previous theorem.
		Let $u\in \partial C$ be nearest to $f(0)$. Further, let $T_u$ be the tangent line at $u$, and $H_u$ the half-plane containing $C$. Then $f\in H(G,H_u)$. Now choose $t$ real so that $(H_u-u)e^{it}=\mathbb{H}$, the right half-plane and let $K=(C-u)e^{it}\subset\mathbb{H}$. Hence
		$$
		g(z)=(f(z)-u)e^{it}\in H(G,K)\subset H(G, \mathbb{H})
		$$
		and
		$$
		g(0)=(f(0)-u)e^{it}=|f(0)-u|=d(f(0),\partial C).
		$$
		Applying Theorem \ref{ThmH} to $g$ provides  the desired conclusion, including sharpness part.
	\end{proof}
	
	\subsection{Multidimensional Bohr's radius for the punctured unit disk $\mathbb{D}_0$}
	Our next objective is to give multidimensional analogue of \cite[Theorem 2.1]{AAN17} for the class $H(G,\mathbb{D}_0)$.
	
	
	\begin{theorem}
		Let $G$ be a bounded balanced domain and $f\in H(G,\mathbb{D}_0)$ be of the form \eqref{eqf} in a neighbourhood of the origin. Then $K^G_X(\mathbb{D}_0)=1/3$.
	\end{theorem}

	\begin{proof}
		We have $H(G,\mathbb{D}_0)\subset H(G,\mathbb{D})$ and thus, the inequality $K^G_X(\mathbb{D}_0)\geq1/3$ directly follows from \cite[Corollary 3.2]{HHK}. To prove $K^G_X(\mathbb{D}_0)\leq1/3$, as before, consider for any $r_0\in (1/3,1)$ a number $0<c<1$ and $v\in \partial G$ such that $cr_0>1/3$ and $c\sup_{x\in \partial G}||x||<||v||$. Next, we introduce
		\begin{equation*}
			H_t(z)=\exp \bigg(-t \cfrac{1+z}{1-z}\bigg)=\cfrac{1}{e^t}+\cfrac{1}{e^t}\sum_{n=1}^{\infty}\bigg[\sum_{m=1}^{n}\cfrac{(-2t)^m}{m!}{n-1 \choose m-1}\bigg]z^n,\quad t>0,
		\end{equation*}
		which is in $H(\mathbb{D},\mathbb{D}_0)$. Also, let $f_1\in H(G,\mathbb{D}_0)$ be defined by
		$$
		f_1(x)=H_t\bigg(\cfrac{c\psi_v(x)}{||v||}\bigg),
		$$
		where $\psi_v$ is a bounded linear functional on $X$ with $\psi_v(v)=||v||$ and $||\psi_v||=1$. Thus, for $x=r_0v$, we get
		$$
		f_1(r_0v)=H_t(cr_0)=\sum_{m=1}^{\infty}\bigg|	\cfrac{\,D^m f_1(0)\big(x^m\big)}{m!}\bigg|>1
		$$
		by using the argument given in \cite[Theorem 2.1]{AAN17}. This concludes the proof.
	\end{proof}
	
	\subsection{Multidimensional Bohr's radius for slit mapping $T$}
	We now state and prove our next result which provides the sharp Bohr inequality for functions belonging to the class $H(G,T)$.
	
	
	\begin{theorem}
		Let $G$ be a bounded balanced domain and $f\in H(G,T)$ be of the form \eqref{eqf} in a neighbourhood of the origin with $f(0)>0$. Then $K^G_X(T)=3-2\sqrt{2}$.
	\end{theorem}

	\begin{proof} By applying \eqref{eqconn}, we can readily show that $K^G_X(T)\geq 3-2\sqrt{2}$. Nevertheless, we present an alternative proof for the lower bound of $K^G_X(T)$.
		Suppose that $f\in H(G,T)$. Now, we consider the function $F$ as in the proof of Theorem \ref{ThmH} which is in $H(\mathbb{D},T)$ having series expansion \eqref{eqHF}. Since $F\in H(\mathbb{D},T)$, we may write the given condition as
		$
		F\prec g
		$,
		where
		$$
		g(z)=f(0)\bigg(\cfrac{1+z}{1-z}\bigg)^2=f(0)+4f(0)\sum_{n=1}^{\infty}nz^n.
		$$
		According to Lemma B(i), 
we have
		\begin{equation*}
			\bigg|	\cfrac{\,D^m f(0)\big(y^m\big)}{m!}\bigg| \le 4f(0)m ~\mbox{ for $m\ge 1$}.
		\end{equation*}
		This gives, for $z\in \mathbb{D}$ and $y\in G$, that
		\begin{equation*}
			\sum_{m=1}^{\infty}\bigg|	\cfrac{\,D^m f(0)\big(y^m\big)}{m!}\bigg| |z|^m\le  4f(0)\sum_{m=1}^{\infty}m |z|^m=4f(0)\cfrac{|z|}{(1-|z|)^2}
		\end{equation*}
		which is less than or equal to $f(0)$ for all $|z|\le 3-2\sqrt{2}\approx0.17157$. The inequality $K^G_X(T)\geq3-2\sqrt{2}$  holds.

		Next we prove that $K^G_X(T)\leq3-2\sqrt{2}$. Note that, for any $r_0\in (3-2\sqrt{2},1)$, there exists a $c\in (0,1)$ such that $cr_0>3-2\sqrt{2}$. Also there exists a $v\in \partial G$ such that
		$$
		c\sup \{||x||:x\in \partial G\}<||v||.
		$$
		Next, we consider a function $f_2$ on $G$ defined by
		$$
		f_2(x)=U\bigg(\cfrac{c\psi_v(x)}{||v||}\bigg),
		$$
		where
		$$
		U(z)=\bigg(\cfrac{1+z}{1-z}\bigg)^2,\quad z\in \mathbb{D}
		$$
		and, $\psi_v$ is a bounded linear functional on $X$ with $\psi_v(v)=||v||$ and $||\psi_v||=1$.
		Clearly $f_2\in H(G,T)$. Then, for $x=r_0v$, we have
		\begin{equation*}
			f_2(x)	=f_2(0)+\sum_{m=1}^{\infty}\bigg|	\cfrac{\,D^m f_2(0)\big(y^m\big)}{m!}\bigg| |z|^m=\bigg(\cfrac{1+cr_0}{1-cr_0}\bigg)^2=1+4\sum_{n=1}^{\infty}n(cr_0)^n>2.
		\end{equation*}
		This completes the proof.
	\end{proof}
	
	\subsection{Multidimensional Bohr's radius for exterior of the closed unit disk $\overline{\mathbb{D}}^c$}
 The following lemma by Abu-Muhanna and Ali \cite{AbuAli} will be required to establish our next result.
	
	\begin{Lem}\label{lemma2.7}
		Let $f\in H(\mathbb{D},\overline{\mathbb{D}}^c)$. Then $f\prec W$, where
		\begin{equation*}
			W(z)=\exp\bigg(\cfrac{1+\phi(z)}{1-\phi(z)}\bigg),
		\end{equation*}
		with
		\begin{equation*}
			\phi(z)=\cfrac{z+b}{1+\overline{z}b} \ \mbox{ and } \ b=\cfrac{\log f(0)-1}{\log f(0)+1}.
		\end{equation*}
	\end{Lem}
	
	The following theorem extends a result of Abu-Muhanna and Ali  \cite[Theorem 2.1]{AbuAli} to  higher dimension.
	\begin{theorem}
		Let $G$ be a bounded balanced domain and $f\in H(G,\overline{\mathbb{D}}^c)$ be of the form \eqref{eqf} in a neighbourhood of the origin. Then
		\begin{equation}\label{eqDc}
			\lambda \bigg(\sum_{m=1}^{\infty}\bigg|	\cfrac{\,D^m f(0)\big(x^m\big)}{m!}\bigg|, |f(0)|\bigg)\le \lambda(f(0),\partial \overline{\mathbb{D}}^c) \quad \mbox{for } x\in (1/3)G,
		\end{equation}
		where $\lambda$ is defined in \eqref{SD}. This result is sharp.
	\end{theorem}
	\begin{proof}
		Let $f\in H(G,\overline{\mathbb{D}}^c)$.
		Then we can construct a function $F$ as in the proof of Theorem \ref{ThmH} which will be in $H(\mathbb{D},\overline{\mathbb{D}}^c)$ and have the series expansion \eqref{eqHF}. Now, by using Lemma F,
we have $F\prec W$.		It is easy to observe that, $W(z)$ defined in Lemma F 
can be written as
		$$
		W(z)=F(0)\exp \bigg(\frac{z\log |F(0)|^2}{1-z}\bigg).
		$$
		Hence, from \cite[Equation 2.7]{AbuAli}, we obtain
		$$
		\sum_{m=0}^{\infty}\bigg|\cfrac{1}{m!}\,D^m f(0)\big(x^m\big)\bigg||z|^m \leq |F(0)|^\frac{1+|z|}{1-|z|}=|f(0)|^\frac{1+|z|}{1-|z|},
		$$
		which, in particular, gives
		$$
		\sum_{m=0}^{\infty}\bigg|\cfrac{1}{m!}\,D^m f(0)\big(x^m\big)\bigg||z|^m \leq |f(0)|^2 \quad \mbox{for } |z|\leq 1/3.
		$$
		Then a simple computation shows that
		$$
		\lambda \bigg(\sum_{m=0}^{\infty}\bigg|\cfrac{1}{m!}\,D^m f(0)\big(x^m\big)\bigg|,|f(0)|\bigg)\leq \lambda(|f(0)|,|f(0)|^2)\leq \lambda(|f(0)|,1)
		$$
		holds for $x\in (1/3)G$.
		
		
		Finally, we prove that inequality \eqref{eqDc} does not hold for $x\in r_0G$, where $r_0\in (1/3,1)$. We know that there exists a $c\in (0,1)$ and $v\in \partial G$ such that $cr_0>1/3$ and
		$$
		c\sup \{||x||:x\in \partial G\}<||v||.
		$$
		Now, we consider $f_3$ on $G$ defined by
		$$
		f_3(x)=W\bigg(\cfrac{c\psi_v(x)}{||v||}\bigg),
		$$
		where $\psi_v$ is a bounded linear functional on $X$ with $\psi_v(v)=||v||$ and $||\psi_v||=1$.
		Thus,
		$$
		f_3(r_0v)=W(cr_0)=|f(0)|\exp \bigg(\frac{cr_0}{1-cr_0}\log |f(0)|^2\bigg)=|f(0)|^{(1+cr_0)/(1-cr_0)}.
		$$
		Also, a simple computation gives that
		$$
		\frac{\lambda(|f(0)|,|f(0)|^{(1+cr_0)/(1-cr_0)})}{\lambda(|f(0)|,1)}= \frac{\sqrt{2}|f(0)|(|f(0)|^{(2cr_0)/(1-cr_0)}-1)}{(|f(0)|-1)\sqrt{1+|f(0)|^{2(1+cr_0)/(1-cr_0)}}}\to \frac{2cr_0}{1-cr_0}
		$$
		as $|f(0)|\to 1$. Since $2x/(1-x)>1$, i.e., $x>1/3$, we see that
		$$
		\lambda \bigg(\sum_{m=0}^{\infty}\bigg|\cfrac{1}{m!}\,D^m f_0(0)\big(x^m\big)\bigg||z|^m,|f(0)|\bigg)= \lambda(|f(0)|,|f(0)|^2)> \lambda(|f(0)|,1)
		$$
		as $|f(0)|\to 1$. This concludes the proof of the theorem.
	\end{proof}

	\noindent
	\subsection*{ Acknowledgement.}
	The work of the first author is
	supported by SERB-SRG, SRG/2023/ 001938, and the work of the second author is supported by the Institute Post Doctoral
	Fellowship of IIT Madras, India. He thanks IIT Madras, for providing an excellent research facility.
	
	\subsection*{Conflict of Interests}
	The authors declare that there is no conflict of interest regarding the publication of this paper.
	
	\subsection*{Data Availability Statement}
	The authors declare that this research is purely theoretical and does not associate with any datas.

\end{document}